\documentclass{amsart}[12pt]

\usepackage{amsmath}
\usepackage{amsfonts}
\usepackage{amssymb}

\makeatletter
\@addtoreset{equation}{section}
\makeatother

\newtheorem{theorem}{Theorem}
\newtheorem{lemma}{Lemma}

\newtheorem{corollary}{Corollary}

\theoremstyle{definition}

\begin{document}
\title[ Invariant subspace of  composition operators]%
{Invariant subspace of composition operators on Hardy space }

\author{Tianyu Bai, Junming Liu*}
\thanks{*Corresponding author}

\address{School of Applied Mathematics, Guangdong University of Technology, Guangzhou, Guangdong, 510520, P.~R.~China}\email{18819799287@163.com}

\address{School of Applied Mathematics, Guangdong University of Technology, Guangzhou, Guangdong, 510520, P.~R.~China}\email{jmliu@gdut.edu.cn}


\begin{abstract}
We consider the invariant subspace of composition operators on Hardy space $H^p$ where the composition operators corresponding to a function $\varphi$ that is a holomorphic self-map of $\mathbb D$. Firstly, we discuss composition operators $C_\varphi$ on subspace $H_{\alpha,\beta}^p$ of Hardy space $H^p$. We will explore the invariant subspaces for $C_\varphi$ in various special cases. Secondly, we consider Beurling type invariant subspace for $C_\varphi$. When $\theta$ is a inner function, we prove that $\theta H^p$ is invariant for $C_\varphi$ if and only if $\displaystyle{\frac{\theta\circ\varphi}{\theta}}$ belongs to $\mathcal S(\mathbb D)$. Thirdly, we obtain that $z^nH^p$ is nontrivial invariant subspace for Deddends algebras $\mathcal D_{C_\varphi}$ when $C_\varphi$ is a compact composition operator and $\varphi$ satisfies that $\varphi(0)=0$ and $\parallel\varphi\parallel_\infty<1$.
\end{abstract}
\thanks{}
\thanks{}
\thanks{This work was supported by NNSF of China (Grant No. 11801094).}
\maketitle

\section{\bf Introduction}
The Hardy space $H^p$ \cite{Z} consists of analytic functions $f$ on open unit disk $\mathbb D$ such that
$$\parallel f\parallel^p_{H^p}=\sup_{0<r<1}\frac{1}{2\pi}\int^{2\pi}_0|f(re^{i\theta})|^pd\theta$$
for $0<p<\infty$. In this paper, we discuss the range of $p$ in Hardy space $H^p$ is $0<p<\infty$.  Let $\varphi$ be a holomorphic self-map of $\mathbb D$. The composition operator $C_\varphi$ with symbol $\varphi$ is the operator defined on Hardy space $H^p$ by
$$C_\varphi f=f\circ\varphi$$
for $z\in\mathbb D$ and $f\in H^p$. And function of the form $C_\varphi^n=f\circ\varphi\circ\cdot\cdot\cdot\circ\varphi$. It would be convenient to write $\varphi_2=\varphi\circ\varphi$ and $\varphi_{n+1}=\varphi_n\circ\varphi$ for $n\in\mathbb N$. The symbol $\mathbb N$ denotes a set of positive integer numbers and $\mathbb N_0$ denotes a set of non-negative integer numbers. Hence $C_\varphi^n=f\circ\varphi_n$. In \cite{Z}, Theorem 11.12 show that
$$\parallel C_\varphi \parallel^p\leq\frac{1+\varphi(0)}{1-\varphi(0)}$$
for $p>0$. It is not hard to see that $C_\varphi$ is a bounded operator on $H^p$. Throughout the section 4 of this paper, we will suppose that $C_\varphi$ is a compact composition operator. For basic nature of composition operator on Hardy space $H^p$, we refer the reader to \cite{ABS,C.C,S.C}.

The invariant subspace problem is one of the most important open problems in linear analysis \cite{H}. Subspaces of Hardy space $H^p$ with unilateral shifts and composition operators have formed an important research field. For studying a family of subspaces of Hardy space for which a sure subalgebra of $H^\infty(\mathbb D)$ is multiplier algebra \cite{DPR}. Here the space $H^\infty(\mathbb D)$ consists of all the functions that are analytic and bounded on $\mathbb D$. The vector operations are the usual pointwise addition of functions and multiplication by complex scalars \cite{MR}. The norm of the function belongs to $H^\infty(\mathbb D)$ is defined by
$$\parallel f\parallel_\infty=\sup\{|f(z)|:z\in\mathbb D\}.$$
It is clear that
$$H^\infty(\mathbb D)\subset H^p\subset H^s$$
for $0<s<p<\infty$ \cite{R}. The corresponding subspace of the Hardy space is $H_{\alpha,\beta}^p$ where $(\alpha,\beta)$ is $admissible$  $pair$. The $admissible$  $pair$ $(\alpha,\beta)$ is a pair complex number that satisfies $|\alpha|^2+|\beta|^2=1$ and $|\alpha|\neq0$. The sequence of functions
$$\{\alpha+\beta z,z^2,z^3,\cdot\cdot\cdot \}$$
is orthonormal basis of subspace $H_{\alpha,\beta}^p$. In fact, $H_{\alpha,\beta}^p$ is the subspace of $H^p$. For specific example $H_{\alpha,\beta}^2$ and some details, we refer the reader to \cite{DPR}. The space $H_0^2$ consists of all functions of Hardy space $H^2$ vanishing at 0, see \cite{S.W}. The space $H_a^2$ consists of all functions of Hardy space $H^2$ vanishing at $a$. We will show that $H_{\alpha,\beta}^p$ is invariant for $C_\varphi$ if and only if $\varphi(z)=z$ for $\beta\neq0$. And other invariant subspaces forms about Hardy space $H_{\alpha,\beta}^p$ of $C_\varphi$ will be verified in this paper.

Denote by $\mathcal S(\mathbb D)$ the set of all holomorphic and bounded  functions by one in modulus on $\mathbb D$, that is,
$$\mathcal S(\mathbb D)=\{f\in H^\infty(\mathbb D):\text{ }\parallel f\parallel_\infty:
=\sup_{z\in\mathbb D}|f(z)|\leq1\}.$$
The set $\mathcal S(\mathbb D)$ is called $Schur$ $class$ and the function $f$ of $\mathcal S(\mathbb D)$ is called $Schur$ $function$. In \cite{B}, A. Beurling proved that every invariant subspace of the unilateral shift operator other than $\{0\}$ has the form $\theta H^2$ where $\theta$ is an inner function. Inner function is function $\theta\in H^\infty(\mathbb D)$ that satisfies $|\theta(z)|\leq1$ for all $z\in\mathbb D$ and $|\theta(e^{it})|=1$ a.e.on $\partial\mathbb D$. In \cite{M}, V. Matache proved that there exist a invariant closed subspace $\theta H^2$ of unilateral shift operator such that subspace $\theta H^2$ is invariant under the composition operator $C_\varphi$ where $\varphi$ is a holomorphic self-map function of $\mathbb D$. In this paper, we will prove that the Beurling type invariant subspace $\theta H^p$ is invariant for $C_\varphi$ if and only if
$$\frac{\theta\circ\varphi}{\theta}\in\mathcal S(\mathbb D).$$
In proof, we will cite Riesz Factorization Theorem \cite{D} as follow:\\
$\textbf{Riesz}$ $\textbf{factorization}$ $\textbf{theorem}$: Let $f$ be a non-zero function in $H^p$($p>0$). Then there exist a Blaschke product $B$  and a function $g\in H^p$ such that
$$g(z)\neq0 \text{ for all } z\in\mathbb D \text{ and } f=Bg.$$
moreover, if $f\in H^\infty(\mathbb D)$, then $g\in H^\infty(\mathbb D)$ and $\parallel f\parallel_\infty=\parallel g\parallel_\infty$. \\
We also will cite Schwartz Lemma \cite{R} as follow:\\
$\textbf{Schwartz}$ $\textbf{Lemma}$: Suppose $f\in H^\infty(\mathbb D)$, $\parallel f\parallel_\infty\leq1$ and $f(0)=0$. Then $|f(z)|\leq|z|$ for all $z\in\mathbb D$ and $|f'(0)|\leq1$. If $|f(z)|=|z|$ for one $z\in\mathbb D-\{0\}$, or $|f'(0)|=1$, then $f(z)=\lambda z$, where $\lambda$ is a constant with $|\lambda|=1$.

Inner function $\theta$ can be decomposed into a Blaschke product $B$ product of a singular inner function $S$, that is
$$\theta=BS.$$
The Blaschke product
$$B(z)=z^m\prod^\infty_{n=1} \frac{\overline{a}_n}{|a_n|}\frac{a_n-z}{1-\overline{a}_nz}$$
for all $z\in\mathbb D$. The sequence $\{a_n\}$ are zeros of $B$ with $\sum_{k=1}^\infty(1-|a_n|)<\infty$ and nonnegative integer $m$ is the multiplicity of $B$ equals to zero at point 0. Blaschke product is inner function. The singular inner function
$$S(z)=\lambda\exp\Big(-\int^{2\pi}_0\frac{e^{it}+z}{e^{it}-z}d\mu(t)\Big)$$
for some unimodular constant $\lambda$ and positive singular measure $\mu$ corresponding to Lebesgue measure.

For Hardy space $H^p$, let $\mathcal L(H^p)$ be the algebra of all bounded linear operators on $H^p$. If $A\in\mathcal L(H^p)$ and there exists a positive integer $M=M(T)>0$ such that
$$\parallel A^nTf\parallel\leq M\parallel A^nf\parallel$$
for all $n\in\mathbb N$ and all $f\in H^p$, then the operator $T$ belongs to Deddens algebra $\mathcal D_A$. $\mathcal D_A$ consists of all operators that satisfy above in-equation. In \cite{D.A}, Deddens puts forward the algebra under the assumption that $A$ is an invertible operator, and he obtain that
$$\sup_{n\in\mathbb N}\parallel A^nTA^{-n}\parallel<\infty.$$
The advantage about our proposed definition is that $A$ can not be an invertible operator. For details and facts about Deddens algebra, we refer the reader to \cite{L,P.D,P.S.R,PS,S.D}. Multiplication operator $M_h$ is defined to be $M_hf(z)=h(z)f(z)$ for $h\in H^{\infty}(\mathbb D)$ and $f\in H^p$. Antiderivative operator $V$ is defined to be $Vf(z)=\int^z_0f(w)dw$ for all $f\in H^p$ and all $z\in\mathbb D$. In the process, we will show that the operators $M_h$ and $V$ belongs to Deddens algebra $\mathcal D_{C_\varphi}$ where ${C_\varphi}$ is a compact composition operator. Then the invariant subspace of Deddens algebra $\mathcal D_{C_\varphi}$ must be of the form $\theta H^p$ where $\theta$ is inner function. In section 4 of this paper, our main purpose is to prove  $z^nH^p$ is nontrivial invariant subspace for Deddends algebras $\mathcal D_{C_\varphi}$ when $\varphi(0)=0$ and $\parallel\varphi\parallel_\infty<1$.

The paper is organized as follows: In Section 2 we discuss composition operators $C_\varphi$ on subspace $H_{\alpha,\beta}^p$ of Hardy space $H^p$. We will show that $JH_{\alpha,\beta}^p$ is invariant for $C_\varphi$ under some special circumstances where $J$ is an inner function. In section 3, the main theorem determining the Beurling type invariant subspace of composition operator $C_\varphi$ on Hardy space $H^p$. Finally, we consider the invariant subspace of Deddens algebra on Hardy space $H^p$.

\section{\bf Invariant subspace of composition operators on subspace $H_{\alpha,\beta}^p$}
The purpose of this section is to discuss the invariant subspace of $C_\varphi$ on subspace $H_{\alpha,\beta}^p$. The following lemma plays a very important role in subsequent proof process and mainly describes some properties of functions on subspace $H_{\alpha,\beta}^p$.

\begin{lemma}\label{lem1}
Let f be a function in $H^p$ and let $(\alpha,\beta)$ be an admissible pair. Then:
(1) The function f is in $H^p_{\alpha,\beta}$ if and only if $f(0)\beta=f'(0)\alpha$.\\
(2) Let J be an inner function such that $J(0)\neq0$, then the function f is in $JH^p_{\alpha,\beta}$ if and only if $J(0)f'(0)\alpha=J(0)f(0)\beta+J'(0)f(0)\alpha$.\\
(3) Let G be an inner function with a zero at the point 0 of multiplicity n where n=0,1,2,\ldots, then f is in $GH^p_{\alpha,\beta}$ if and only if
$$f(0)=f'(0)=\cdots=f^{(n-1)}(0)=0$$ and
$$f^{(n+1)}(0)=\Big\{(n+1)\frac{\beta}{\alpha}+\frac{G^{(n+1)}(0)}{G^{(n)}(0)}\Big\}f^{(n)}(0).$$
\end{lemma}

\begin{proof}
 For part 1, suppose the function $f\in H^p_{\alpha,\beta}$, then
 \begin{equation}\label{eq1}
 f(z)=a_{1}(\alpha+\beta z)+a_{2}z^2+a_{3}z^3+\cdots \text{ for  }z\in \mathbb D,
 \end{equation}
 hence $f(0)=a_{1}\alpha$ and $f'(0)=a_{1}\beta$. Therefore, we get that $f(0)\beta=f'(0)\alpha$. Conversely, it is not hard to see that $f$ has a power series expansion and satisfies (\ref{eq1}). If $f$ satisfies $f(0)\beta=f'(0)\alpha$, then $a_{1}=\displaystyle{\frac{f(0)}{\alpha}}=\displaystyle{\frac{f'(0)}{\beta}}$ when $\beta\neq0$ and $a_{1}=f(0)$ when $\beta=0$.

 For part 2, the function $f$ belongs to $JH^p_{\alpha,\beta}$ if and only if $f=Jg$ for some $g\in H^p_{\alpha,\beta}$, which is equivalent to $\displaystyle{\frac{f}{J}}\in H^p_{\alpha,\beta}$. And $\displaystyle{\frac{f}{J}\in H^p_{\alpha,\beta}}$ if and only if $\big(\displaystyle{\frac{f}{J}\big)(0)\beta=\big(\frac{f}{J}\big)'(0)\alpha}$, which is equivalent to $J(0)f'(0)\alpha=J(0)f(0)\beta+J'(0)f(0)\alpha$.

 For the part 3, let $G(z)=z^nJ(z)$ where $J$ is an inner function and $J(0)\neq0$, then $G^{(n+1)}(0)=(n+1)!J'(0)$ and $G^{(n)}(0)=n!J(0)$, so $\displaystyle{\frac{G^{(n+1)}(0)}{G^{(n)}(0)}=(n+1)\frac{J'(0)}{J(0)}}$. If $f\in GH^p_{\alpha,\beta}$, suppose $f(z)=z^nJ(z)g_1(z)$ for some $g_{1}\in H^p_{\alpha,\beta}$, then
 $$g_{1}(z)=b_{1}(\alpha+\beta z)+b_{2}z^2+b_{3}z^3+\cdots,$$
 then
 $$J(z)g_{1}(z)=b_{1}(J(z)\alpha+J(z)\beta z)+J(z)b_{2}z^2+J(z)b_{3}z^3+\cdots.$$
 Therefore,
 $$f(z)=b_{1}(J(z)\alpha z^n+J(z)\beta z^{n+1})+J(z)b_{2}z^{n+2}+J(z)b_{3}z^{n+3}+\cdots.$$
 It is not hard to see that
 $$f(0)=f'(0)=\cdots=f^{(n-1)}(0)=0.$$
 Taking the $n$-derivative and $n+1$-derivative of $f$ at the zero point respectively, we obtain that
 \begin{equation}\label{eq2}
 f^{(n)}(0)=n!b_1J(0)\alpha
 \end{equation}
 and
 \begin{equation}\label{eq3}
 f^{(n+1)}(0)=(n+1)!\big(b_1J'(0)\alpha+b_1J(0)\beta\big),
 \end{equation}
 multiplying both sides of (\ref{eq3}) by $J(0)\alpha$, and then combine with (\ref{eq2}) to get
 $$J(0)\alpha f^{(n+1)}(0)=(n+1)\big(J'(0)\alpha+J(0)\beta\big)f^{(n)}(0),$$
 then
 $$f^{(n+1)}(0)=\Big\{(n+1)\displaystyle{\frac{\beta}{\alpha}+(n+1)\frac{J'(0)}{J(0)}\Big\}f^{(n)}(0)}.$$ Therefore,
 $$f^{(n+1)}(0)=\Big\{(n+1)\displaystyle{\frac{\beta}{\alpha}+
 \frac{G^{(n+1)}(0)}{G^{(n)}(0)}\Big\}f^{(n)}(0)}.$$
 Conversely, if $f\in H^p$ and $f(0)=f'(0)=\cdots=f^{(n-1)}(0)=0$, there exist a function $g_{2}\in H^p$ such that $f(z)=z^ng_{2}(z)$, then we obtain that
 $$f^{(n)}(0)=n!g_2(0) \text{ and } f^{(n+1)}(0)=(n+1)!g'_2(0).$$
 And the condition
 $$f^{(n+1)}(0)=\Big\{(n+1)\frac{\beta}{\alpha}+\frac{G^{(n+1)}(0)}{G^{(n)}(0)}\Big\}f^{(n)}(0)$$ is equivalent to
 $$J(0)\alpha f^{(n+1)}(0)=(n+1)\big(J'(0)\alpha+J(0)\beta\big)f^{(n)}(0),$$
 then
 $$J(0)\alpha(n+1)!g_{2}'(0)=(n+1)\big(J'(0)\alpha+J(0)\beta\big)n!g_{2}(0).$$
 Therefore
 $$J(0)g_{2}'(0)\alpha=J(0)g_2(0)\beta+ J'(0)g_{2}(0)\alpha.$$
 Thus $g_2\in JH^p_{\alpha,\beta}$, and hence $f(z)=z^ng_2(z)\in z^nJH^p_{\alpha,\beta}=GH^p_{\alpha,\beta}.$
\end{proof}

Next, we will explore two special cases, the first is $\beta\neq0$, the second is $\alpha=1$ and $\beta=0$. The purpose of the following lemma is to verify under what conditions, the subspace $H^p_{\alpha,\beta}$ that satisfies the first two cases is invariant for $C_\varphi$.

\begin{lemma}\label{lem2}
Let $\varphi$ be a holomorphic self-map of $\mathbb D$  and let $(\alpha,\beta)$ be an admissible pair. Then:\\
(1) For $\beta\neq0$, $C_{\varphi}(H^p_{\alpha,\beta})\subseteq H^p_{\alpha,\beta}$ if and only if $\varphi(z)=z$.\\
(2) $C_{\varphi}(H^p_{1,0})\subseteq H^p_{1,0}$ if and only if either $\varphi(0)=0$ or $\varphi'(0)=0$.
\end{lemma}

\begin{proof}
 For part 1, suppose $C_{\varphi}(H^p_{\alpha,\beta})\subseteq H^p_{\alpha,\beta}$, then $C_\varphi f=f(\varphi(z))\in H^p_{\alpha,\beta}$ for any $f\in H^p_{\alpha,\beta}$. By the Lemma \ref{lem1}(1), we obtain that
 \begin{equation}\label{eq4}
 f\big(\varphi(0)\big)\beta=f'\big(\varphi(0)\big)\varphi'(0)\alpha .
 \end{equation}
 Taking $f(z)=z^n$ ($n\geq2$) in (\ref{eq4}), we get
 \begin{equation}\label{eq5}
 \big(\varphi(0)\big)^n\beta=n\big(\varphi(0)\big)^{n-1}\varphi'(0)\alpha,
 \end{equation}
 Assume $\varphi(0)\neq0$, then $\varphi(0)\beta=n\varphi'(0)\alpha$. So $\varphi'(0)=\displaystyle{\frac{\varphi(0)\beta}{n\alpha}}$ for $n\geq2$. Taking $n\rightarrow\infty$, we get $\varphi'(0)=0$. Taking $\varphi'(0)=0$ substitute into (\ref{eq5}), we get $\varphi(0)=0$. The contradiction shows that $\varphi(0)=0$.
 And then taking $f(z)=\alpha+\beta z$ in (\ref{eq4}), we obtain that $\big(\alpha+\beta\varphi(0)\big)\beta=\varphi'(0)\alpha\beta$. So $\varphi'(0)=1$. Therefore, by the Schwartz Lemma, we get $\varphi(z)=z$.
 Conversely, suppose $\varphi(z)=z$, we obtain that $C_\varphi f=f\in H^p_{\alpha,\beta}$ for any $f\in H^p_{\alpha,\beta}$. Thus $C_{\varphi}(H^p_{\alpha,\beta})\subseteq H^p_{\alpha,\beta}$.

 For part 2, suppose $C_{\varphi}(H^p_{1,0})\subseteq H^p_{1,0}$, then $C_\varphi f\in H^p_{1,0}$ for any $f\in H^p_{1,0}$. Then we obtain that $f'\big(\varphi(0)\big)\varphi'(0)=0$ by Lemma \ref{lem1}(1), taking $f(z)=z^2$ in it, we get either $\varphi(0)=0$ or $\varphi'(0)=0$. Conversely, for any $f\in H^p_{1,0}$, then $f'(0)=0$ by Lemma \ref{lem1}(1). If $\varphi(0)=0$, for $\alpha=1$ and $\beta=0$, we get
 $$(C_\varphi f)(0)\beta=f(\varphi(0))\beta=0=f'(0)\varphi'(0)\alpha=f'(\varphi(0))\varphi'(0)\alpha,$$ we obtain that $C_\varphi f\in H_{1,0}^p$.
 If $\varphi'(0)=0$, for $\alpha=1$ and $\beta=0$, we get
 $$(C_\varphi f)(0)\beta=f(\varphi(0))\beta=0=f'(\varphi(0))\varphi'(0)\alpha,$$
 it is not hard to see that $C_\varphi f\in H_{1,0}^p$. Thus, the condition that either $\varphi(0)=0$ or $\varphi'(0)=0$ can imply $C_\varphi (H_{0,1}^p)\in H_{1,0}^p$.
\end{proof}

The following lemma is useful.

\begin{lemma}\label{lem3}
 Let J be an inner function such that $J(0)\neq0$. Then for an admissible pair $(\alpha,\beta)$, $J(0)\beta+J'(0)\alpha=0$ if and only if $JH^p_{\alpha,\beta}=H^p_{1,0}$. Further, if J has a zero at the point 0 of multiplicity n, then $(n+1)\displaystyle{\frac{\beta}{\alpha}}+\displaystyle{\frac{J^{(n+1)}(0)}{J^{(n)}(0)}}=0$ if and only if $JH^p_{\alpha,\beta}=z^nH^p_{1,0}$.
\end{lemma}

\begin{proof}
 For the first half, suppose $J(0)\beta+J'(0)\alpha=0$. Notice that $f\in JH^p_{\alpha,\beta}$ if and only if
 $$J(0)f'(0)\alpha=J(0)f(0)\beta+J'(0)f(0)\alpha.$$
 Then $f'(0)=0$ since $J(0)\beta+J'(0)\alpha=0$ and $J(0)\neq0$. By Lemma \ref{lem1}(1), we get $f\in H^p_{1,0}$. So $f\in JH^p_{\alpha,\beta}$ if and only if $f\in H^p_{1,0}$. Therefore, $JH^p_{\alpha,\beta}=H^p_{1,0}$. Conversely, suppose $JH^p_{\alpha,\beta}=H^p_{1,0}$, let $g(z)=(\alpha+\beta z)\in H^p_{\alpha,\beta}$, it is not hard to see that $Jg\in H^p_{1,0}$, then $(Jg)'(0)=0$, then $J(0)\beta+J'(0)\alpha=0$.

 For the latter part, we know that $J$ has a zero at the point 0 of multiplicity $n$. Then let $J(z)=z^nJ_1(z)$ and $J_1(0)\neq0$, we get $J^{(n)}(0)=n!J_1(0)$ and $J^{(n+1)}(0)=(n+1)!J'_1(0)$. Hence $(n+1)\displaystyle{\frac{\beta}{\alpha}+\frac{J^{(n+1)}(0)}{J^{(n)}(0)}}=0$ is equivalent to $$(n+1)\Big\{\frac{\beta}{\alpha}+\frac{J'_1(0)}{J_1(0)}\Big\}=0,$$
 then $J_1(0)\beta+J'_1(0)\alpha=0$ if and only if $J_1H^p_{\alpha,\beta}=H^p_{1,0}$ if and only if $z^nJ_1H^p_{\alpha,\beta}=z^nH^p_{1,0}$, that is, $JH^p_{\alpha,\beta}=z^nH^p_{1,0}$.
\end{proof}

The next two theorems give the conditions of $JH_{\alpha,\beta}^p$ under composition operator $C_\varphi$ for special types of function $J$ and $\varphi$.

\begin{theorem}\label{the1}
Let $\beta\neq0$ and $\varphi(z)=z^k$.\\
(1) If $n=1 $ and $k\in\mathbb N\setminus\{2\}$, then the subspace $zH^p_{\alpha,\beta}$ is invariant under $C_\varphi$.\\
(2) If $n\geq2$ and $k\in\mathbb N$, then the subspace $z^nH^p_{\alpha,\beta}$ is invariant under $C_\varphi$
\end{theorem}

\begin{proof}
 For part 1, suppose $k=1$, then $\varphi(z)=z$. It is not hard to see that $C_\varphi(zH^p_{\alpha,\beta})\subseteq zH^p_{\alpha,\beta}$. Suppose $k\geq3$. Let $f(z)=zg(z)\in zH^p_{\alpha,\beta}$ for $g\in H^p_{\alpha,\beta}$, then
 $$g(z)=a_1(\alpha+\beta z)+a_2z^2+a_3z^3+\cdots,$$
 then
 $$f(z)=z\{a_1(\alpha+\beta z)+a_2z^2+a_3z^3+\cdots\}.$$
 Therefore
 \begin{align*}
 (C_\varphi f)(z)=&f(z^k)\\
 =&z^k\{a_1(\alpha+\beta z^k)+a_2z^{2k}+a_3z^{3k}+\cdots\}\\
 =&z\{a_1(\alpha z^{k-1}+\beta z^{2k-1})+a_2z^{3k-1}+a_3z^{4k-1}+\cdots\}.
 \end{align*}
 Hence $C_\varphi f\in zH^p_{\alpha,\beta}$, as $k\geq3$. Therefore  $C_\varphi(zH^p_{\alpha,\beta})\subseteq zH^p_{\alpha,\beta}$.

 For part 2, suppose $n\geq2$ and $k\geq1$. Let $f(z)=z^ng(z)\in z^nH^p_{\alpha,\beta}$ for $g\in H^p_{\alpha,\beta}$, then
 $$g(z)=b_1(\alpha+\beta z)+b_2z^2+b_3z^3+\cdots,$$ then
 $$f(z)=z^n\{b_1(\alpha+\beta z)+b_2z^2+b_3z^3+\cdots\}.$$ Therefore,
 \begin{align*}
 (C_\varphi f)(z)=&f(z^k)\\
 =&z^{nk}\{b_1(\alpha+\beta z^k)+b_2z^{2k}+b_3z^{3k}+\cdots\}\\
 =&z^{n}\{b_1(\alpha z^{(k-1)n}+\beta z^{(k-1)n+k})+b_2z^{(k-1)n+2k}+b_3z^{(k-1)n+3k}+\cdots\}.
 \end{align*}
 Thus $C_\varphi f\in z^nH^p_{\alpha,\beta}$. Therefore, $C_\varphi(z^nH^p_{\alpha,\beta})\subseteq z^nH^p_{\alpha,\beta}$.
\end{proof}

 We consider whether the subspace $zH^p_{\alpha,\beta}$ is invariant under $C_\varphi$ for $\varphi(z)=z^2$. Let $g(z)=(\alpha+\beta z)$, then $f(z)=zg(z)=z(\alpha+\beta z)\in zH^p_{\alpha,\beta}$. Then
 $$(C_\varphi f)(z)=C_\varphi\big(z(\alpha+\beta z)\big)=z^2(\alpha+\beta z^2)=z(\alpha z+\beta z^3)\notin zH^p_{\alpha,\beta}.$$
 Therefore, $zH^p_{\alpha,\beta}$ is not invariant under $C_\varphi$ for $\varphi(z)=z^2$.

\begin{theorem}\label{the2}
 Let $\varphi$ be a holomorphic self-map of $\mathbb D$ and $\lambda>0$. Then:\\
 (1) If $\lambda\neq\frac{\beta}{2\alpha}$, then $e^{\lambda\frac{z+1}{z-1}}H^p_{\alpha,\beta}$ is invariant under $C_\varphi$ if and only if $\varphi(z)=z$.\\
 (2)If $\lambda=\frac{\beta}{2\alpha}$, then $e^{\lambda\frac{z+1}{z-1}}H^p_{\alpha,\beta}$ is invariant under $C_\varphi$ if and only if either $\varphi(0)=0$ or $\varphi'(0)=0$.
\end{theorem}

\begin{proof}
 Notice that $e^{\lambda\frac{z+1}{z-1}}f\in e^{\lambda\frac{z+1}{z-1}}H^p_{\alpha,\beta}$ for any $f\in H^p_{\alpha,\beta}$, then
 $$C_\varphi\big(e^{\lambda\frac{z+1}{z-1}}f(z)\big)
 =e^{\lambda\frac{\varphi(z)+1}{\varphi(z)-1}}f\big(\varphi(z)\big).$$
 Let $g(z)=e^{\lambda(\frac{\varphi(z)+1}{\varphi(z)-1}-\frac{z+1}{z-1})}f\big(\varphi(z)\big)$, then
 $$e^{\lambda\frac{z+1}{z-1}}g(z)=e^{\lambda\frac{\varphi(z)+1}{\varphi(z)-1}}f\big(\varphi(z)\big)
 =C_\varphi\big(e^{\lambda\frac{z+1}{z-1}}f(z)\big).$$
 Thus, $g(z)\in H^p_{\alpha,\beta}$ if and only if $e^{\lambda\frac{z+1}{z-1}}g(z)\in e^{\lambda\frac{z+1}{z-1}}H^p_{\alpha,\beta}$ if and only if $C_\varphi(e^{\lambda\frac{z+1}{z-1}}f(z))\in e^{\lambda\frac{z+1}{z-1}}H^p_{\alpha,\beta}$ if and only if $C_\varphi(e^{\lambda\frac{z+1}{z-1}}H^p_{\alpha,\beta})\subseteq e^{\lambda\frac{z+1}{z-1}} H^p_{\alpha,\beta}$.

 Suppose $e^{\lambda\frac{z+1}{z-1}}H^p_{\alpha,\beta}$ is invariant under $C_\varphi$, that is  $C_\varphi(e^{\lambda\frac{z+1}{z-1}}H^p_{\alpha,\beta})\subseteq e^{\lambda\frac{z+1}{z-1}} H^p_{\alpha,\beta}$, which implies that $g\in H^p_{\alpha,\beta}$. Then $g(0)\beta=g'(0)\alpha$ by Lemma \ref{lem1}(1). It is not hard to see that
 \begin{equation}\label{eq6}
 g(0)=e^{2\lambda(\frac{\varphi(0)}{\varphi(0)-1})}f(\varphi(0))
 \end{equation}
 and
 $$g'(z)=e^{\lambda(\frac{\varphi(z)+1}{\varphi(z)-1}-\frac{z+1}{z-1})}
 \Big[2\lambda\Big(\frac{1}{(z-1)^2}-\frac{\varphi'(z)}{(\varphi(z)-1)^2}\Big)f(\varphi(z))
 +f'(\varphi(z))\varphi'(z)\Big],$$
 then
 \begin{equation}\label{eq7}
 g'(0)=e^{2\lambda\big(\frac{\varphi(0)}{\varphi(0)-1}\big)}
 \Big[2\lambda\Big(1-\frac{\varphi'(0)}{(\varphi(0)-1)^2}\Big)f(\varphi(0))
 +f'(\varphi(0))\varphi'(0)\Big].
 \end{equation}
 So substituting (\ref{eq6}) and (\ref{eq7}) into $g(0)\beta=g'(0)\alpha$, we get
 \begin{equation}\label{eq8}
 f(\varphi(0))\beta
 =\Big[2\lambda\Big(1-\frac{\varphi'(0)}{(\varphi(0)-1)^2}\Big)f(\varphi(0))
 +f'(\varphi(0))\varphi'(0)\Big]\alpha.
 \end{equation}

 Firstly, for $\lambda\neq\frac{\beta}{2\alpha}$, taking $f(z)=z^n$ ($n\geq2$) in (\ref{eq8}), we get
 \begin{equation}\label{eq9}
 (\varphi(0))^n\beta=\Big[2\lambda\Big(1-\frac{\varphi'(0)}{\big(\varphi(0)-1\big)^2}\Big)(\varphi(0))^n
 +n(\varphi(0))^{(n-1)}\varphi'(0)\Big]\alpha.
 \end{equation}
 Assume $\varphi(0)\neq0$, then
 \begin{equation}\label{eq10}
 \frac{1}{n}\varphi(0)\beta=\Big[\frac{1}{n}2\lambda\Big(1-\frac{\varphi'(0)}{(\varphi(0)-1)^2}\Big)\varphi(0)
 +\varphi'(0)\Big]\alpha.
 \end{equation}
 Taking $n\rightarrow\infty$, then $\varphi'(0)=0$. Let $\varphi'(0)=0$ take in (\ref{eq9}), we get
 $$\varphi(0)\beta=\varphi(0)2\lambda\alpha.$$
 Then $\varphi(0)=0$ since $\lambda\neq\frac{\beta}{2\alpha}$. The contradiction shows that $\varphi(0)=0$. Therefore,
 \begin{equation}\label{eq11}
 f(0)\beta=\Big[2\lambda\big(1-\varphi'(0))f(0)+f'(0)\varphi'(0)\Big]\alpha.
 \end{equation}
 And then taking $f(z)=\alpha+\beta z$ in (\ref{eq11}), we get
 $$\beta=2\lambda\alpha\big(1-\varphi'(0)\big)+\beta\varphi'(0),$$
 which implies $\varphi'(0)=1$ since $\lambda\neq\frac{\beta}{2\alpha}$. According to the Schwartz Lemma, we get that $\varphi(z)=z$.

 After that, for $\lambda=\frac{\beta}{2\alpha}$, taking it in (\ref{eq8}), then
 \begin{equation}\label{eq12}
 f(\varphi(0))\beta
 =\Big(1-\frac{\varphi'(0)}{(\varphi(0)-1)^2}\Big)f(\varphi(0))\beta
 +f'(\varphi(0))\varphi'(0)\alpha.
 \end{equation}
 Taking $f(z)=z^n$ ($n\geq2$) in (\ref{eq12}), we get
 \begin{equation}\label{eq13}
 (\varphi(0))^n\beta
 =\Big(1-\frac{\varphi'(0)}{(\varphi(0)-1)^2}\Big)\big(\varphi(0)\big)^n\beta
 +n\big(\varphi(0)\big)^{n-1}\varphi'(0)\alpha.
 \end{equation}
 If $\varphi(0)\neq0$, we have
 \begin{equation}\label{eq14}
 \frac{\beta}{n}\varphi(0)
 =\Big(1-\frac{\varphi'(0)}{(\varphi(0)-1)^2}\Big)\varphi(0)\frac{\beta}{n}
 +\varphi'(0)\alpha.
 \end{equation}
 Taking $n\rightarrow\infty$, we get $\varphi'(0)=0$.

 Conversely, for $\lambda\neq\frac{\beta}{2\alpha}$. Suppose $\varphi(z)=z$, then $C_\varphi(e^{\lambda\frac{z+1}{z-1}}f)=e^{\lambda\frac{z+1}{z-1}}f\in e^{\lambda\frac{z+1}{z-1}}H^p_{\alpha,\beta}$ for any $f\in H^p_{\alpha,\beta}$. Thus
 $C_\varphi(e^{\lambda\frac{z+1}{z-1}}H^p_{\alpha,\beta})\subseteq e^{\lambda\frac{z+1}{z-1}}H^p_{\alpha,\beta}$. For $\lambda=\frac{\beta}{2\alpha}$.
 Suppose $\varphi(0)=0$, then $g(0)=f(0)$ and $g'(0)=2\lambda f(0)(1-\varphi'(0))+f'(\varphi(0))\varphi'(0)$. Therefore,
 \begin{align*}
 \alpha g'(0)=&2\alpha\lambda f(0)(1-\varphi'(0))+\alpha f'(\varphi(0))\varphi'(0)\\
 =&\beta f(0)(1-\varphi'(0))+\alpha f'(0)\varphi'(0)\\
 =&\beta f(0)\\
 =&\beta g(0)
 \end{align*}
 Hence, $g\in H^p_{\alpha,\beta}$, then we get $C_\varphi(e^{\lambda\frac{z+1}{z-1}}H^p_{\alpha,\beta})\subseteq H^p_{\alpha,\beta}$. If $\varphi'(0)=0$, then
 \begin{align*}
 \alpha g'(0)=&e^{2\lambda(\frac{\varphi(0)}{\varphi(0)-1})}2\lambda \alpha f(\varphi(0))\\
 =&\beta e^{2\lambda(\frac{\varphi(0)}{\varphi(0)-1})}f(\varphi(0))\\
 =&\beta g(0)
 \end{align*}
 Therefore $g\in H^p_{\alpha,\beta}$, then we get $C_\varphi(e^{\lambda\frac{z+1}{z-1}}H^p_{\alpha,\beta})\subseteq H^p_{\alpha,\beta}$.
\end{proof}

Theorems \ref{the1} and \ref{the2} are examples of more general results about the invariance of $JH^p_{\alpha,\beta}$ under $C_\varphi$. Due to the intuitive value, we choose to keep special cases as independent results.

Next we will pay attention to the equivalent conditions of $JH^p_{\alpha,\beta}$ as an invariant subspace under $C_\varphi$ in general.

\begin{theorem}\label{the3}
 Let $\varphi:\mathbb D\longrightarrow\mathbb D$ be a holomorphic and different from the identity map, $(\alpha,\beta)$ be an admissible pair and J be an inner function such that $JH^p_{\alpha,\beta}\neq H^p_{1,0}$. Suppose that $0\notin \mathcal Z_J$, the set of zeros of J. Then
 $$C_\varphi(JH^p_{\alpha,\beta})\subseteq JH^p_{\alpha,\beta} \text{ if and only if } \varphi(\mathcal Z_J)\subseteq \mathcal Z_J \text{ and } J\circ\varphi\in z^2H^\infty.$$
\end{theorem}

\begin{proof}
 Suppose that $C_\varphi (JH^p_{\alpha,\beta})\subseteq JH^p_{\alpha,\beta}$.

 Firstly, we will show that $\varphi(\mathcal Z_J)\subseteq \mathcal Z_J$. Suppose to the contrary that $\varphi(a)\notin \mathcal Z_J$ for some $a\in \mathcal Z_J$, that is, $J\big(\varphi(a)\big)\neq0$. Let $f=Jg\in JH^p_{\alpha,\beta}$ for $g\in H^p_{\alpha,\beta}$, then $C_\varphi f\in JH^p_{\alpha,\beta}$. Then
 $$(C_\varphi f)(a)=f\big(\varphi(a)\big)=J\big(\varphi(a)\big)g\big(\varphi(a)\big)=0.$$
 Thus $g\big(\varphi(a)\big)=0$ since $J\big(\varphi(a)\big)\neq0$. Taking $g(z)=\alpha+\beta z$, we obtain that
 $$\alpha+\beta\varphi(a)=0.$$
 Assume $\beta=0$, we get $\alpha=0$, then $(\alpha,\beta)$ is not a admissible pair. Hence $\beta\neq0$ and $\varphi(a)=-\displaystyle{\frac{\alpha}{\beta}}$. And taking $g(z)=z^2$, then $\displaystyle{\frac{\alpha^2}{\beta^2}}=0$ since $g\big(\varphi(a)\big)=g(-\displaystyle{\frac{\alpha}{\beta})}=0$. Thus $\alpha=0$, a contradiction. Therefore, $J\big(\varphi(a)\big)=0$, that is $\varphi(\mathcal Z_J)\subseteq \mathcal Z_J$.

 Next, we will show that $J\circ\varphi\in z^2H^\infty$. Let $f=Jg\in H^p_{\alpha,\beta}$ for $g\in H^p_{\alpha,\beta}$, then $C_\varphi f\in JH^p_{\alpha,\beta}$. According to Lemma \ref{lem1}(2), we get
 $$J(0)(f\circ\varphi)'(0)\alpha=J(0)(f\circ\varphi)(0)\beta+J'(0)(f\circ\varphi)(0)\alpha, $$
 which is equivalent to
 \begin{equation}\label{eq15}
 (f\circ\varphi)'(0)=\Big\{\frac{\beta}{\alpha}+\frac{J'(0)}{J(0)}\Big\}(f\circ\varphi)(0).
 \end{equation}
 And then
 \begin{equation}\label{eq16}
 \Big[J'\big(\varphi(0)\big)g\big(\varphi(0)\big)+J\big(\varphi(0)\big)g'\big(\varphi(0)\big)\Big]\varphi'(0)
 =\Big\{\frac{\beta}{\alpha}+\frac{J'(0)}{J(0)}\Big\}J(\varphi(0))g(\varphi(0)).
 \end{equation}
 Assume $\varphi(0)=0$, then
 \begin{equation}\label{eq17}
 \Big[J'(0)g(0)+J(0)g'(0)\Big]\varphi'(0)
 =\Big\{\frac{\beta}{\alpha}+\frac{J'(0)}{J(0)}\Big\}J(0)g(0).
 \end{equation}
 Taking $g(z)=\alpha+\beta z$ in (\ref{eq17}), then
 $$\Big[J'(0)\alpha+J(0)\beta\Big]\varphi'(0)
 =\Big\{\frac{\beta}{\alpha}+\frac{J'(0)}{J(0)}\Big\}J(0)\alpha.$$
 Thus $\varphi'(0)=1$ since $J'(0)\alpha+J(0)\beta\neq0$. According to the Schwartz Lemma, we get $\varphi(z)=z$. The contradiction shows that $\varphi(0)\neq0$. And then taking $g(z)=z^n$ ($n\geq2$) in (\ref{eq16}), we get
 $$\Big[J'\big(\varphi(0)\big)\big(\varphi(0)\big)^n+J\big(\varphi(0)\big)n\big(\varphi(0)\big)^{n-1}
 \Big]\varphi'(0)
 =\Big\{\frac{\beta}{\alpha}+\frac{J'(0)}{J(0)}\Big\}J\big(\varphi(0)\big)\big(\varphi(0)\big)^n.$$
 Then
 \begin{equation}\label{eq18}
 \Big[\frac{1}{n}J'\big(\varphi(0)\big)\varphi(0)+J\big(\varphi(0)\big)\Big]\varphi'(0)
 =\frac{1}{n}\Big\{\frac{\beta}{\alpha}+\frac{J'(0)}{J(0)}\Big\}J\big(\varphi(0)\big)\varphi(0).
 \end{equation}
 Taking $n\rightarrow\infty$, we get $J(\varphi(0))\varphi'(0)=0$. Assume $J(\varphi(0))\neq0$, then $\varphi'(0)=0$. Taking $\varphi'(0)=0$ in (\ref{eq18}), we obtain that
 $$\frac{1}{n}\Big\{\frac{\beta}{\alpha}+\frac{J'(0)}{J(0)}\Big\}J\big(\varphi(0)\big)\varphi(0)=0.$$
 Then $J(\varphi(0))=0$. The contradiction shows that $J(\varphi(0))=0$. Taking $J(\varphi(0))=0$ in (\ref{eq18}), we get $J'(\varphi(0))\varphi'(0)=0$, that is $(J\circ\varphi)'(0)=0$. And $J$ is a inner function. Hence $J\circ\varphi\in z^2H^\infty$.

 Conversely, suppose $\varphi(\mathcal Z_J)\subseteq \mathcal Z_J \text{ and } J\circ\varphi\in z^2H^\infty$. Hence
 $$J(\varphi(0))=0 \text{ and } J'(\varphi(0))\varphi'(0)=0.$$
 Let $f=Jg\in JH^p_{\alpha,\beta}$ for $g\in H^p_{\alpha,\beta}$, then
 $$(f\circ\varphi)(0)=J\big(\varphi(0)\big)g\big(\varphi(0)\big)=0$$ and
 $$(f\circ\varphi)'(0)=J'\big(\varphi(0)\big)g\big(\varphi(0)\big)\varphi'(0)+
 J\big(\varphi(0)\big)g'\big(\varphi(0)\big)\varphi'(0)=0.$$
 Therefore,
 $$(f\circ\varphi)'(0)=\Big\{\frac{\beta}{\alpha}+\frac{J'(0)}{J(0)}\Big\}(f\circ\varphi)(0)=0,$$ that is,
 $$J(0)(f\circ\varphi)'(0)\alpha=J(0)(f\circ\varphi)(0)\beta+J'(0)(f\circ\varphi)(0)\alpha=0.$$
 By Lemma \ref{lem1}(2), we get $C_\varphi f\in JH^p_{\alpha,\beta}$. Hence $C_\varphi JH^p_{\alpha,\beta}\subseteq JH^p_{\alpha,\beta}$.
\end{proof}

In the next theorem, we will discuss the case that 0 belongs to the set $\mathcal Z_J$.

\begin{theorem}\label{the4}
 Let $\varphi:\mathbb D\longrightarrow\mathbb D$ be a holomorphic and different from the identity map, $(\alpha,\beta)$ be an admissible pair, and $J$ an inner function such that $0\in\mathcal Z_J$ with multiplicity n. If $\varphi(\mathcal Z_J)\subseteq \mathcal Z_J \text{ and } J\circ\varphi\in z^{n+2}H^\infty$, then $C_\varphi(JH^p_{\alpha,\beta})\subseteq JH^p_{\alpha,\beta}$.
\end{theorem}

\begin{proof}
 If $\varphi(\mathcal Z_J)\subseteq \mathcal Z_J \text{ and } J\circ\varphi\in z^{n+2}H^\infty$, then
 $$(J\circ\varphi)(0)=(J\circ\varphi)'(0)=\cdots=(J\circ\varphi)^{(n+1)}(0)=0.$$
 Let $f=Jg\in H^p_{\alpha,\beta}$ for $g\in H^p_{\alpha,\beta}$, then
 \begin{align*}
 (f\circ\varphi)^{(k)}=&\Big((J\circ\varphi)(g\circ\varphi)\Big)^{(k)}(0)\\
 =&\sum^k_{m=0}C^m_k(J\circ\varphi)^{(k-m)}(0)(g\circ\varphi)^{(m)}(0).
 \end{align*}
 Then $(f\circ\varphi)^{(k)}(0)=0$ for $0\leq k\leq n+1$. Therefore,
 $$(f\circ\varphi)^{(n+1)}(0)
 =\Big\{(n+1)\frac{\beta}{\alpha}+\frac{J^{(n+1)}(0)}{J^{(n)}(0)}\Big\}(f\circ\varphi)^{(n)}(0)=0.$$
 By Lemma \ref{lem1}(3), we get that $C_\varphi f=(f\circ\varphi)\in JH^p_{\alpha,\beta}$. Hence $C_\varphi (JH^p_{\alpha,\beta})\subseteq JH^p_{\alpha,\beta}$.
\end{proof}

By \cite{S.C.}, Schwartz states that a holomorphic self-map of $\mathbb D$ is an automorphism if and only if $C_\varphi$ is invertible operator that belongs to $H^2$. In the following theorem, we discuss the range will narrow to the subspace of $H^2$.

\begin{theorem}\label{the5}
Let $\varphi:\mathbb D\longrightarrow\mathbb D$ be a holomorphic function. Let $H^2_a$ denote the subspace of $H^2$ consisting of functions vanishing at $a\in\mathbb D$ and let $H^2_b$ denote the subspace of $H^2$ consisting of functions vanishing at $b\in\mathbb D$. Then $C_\varphi: H_a^2\rightarrow H^2_b$ is invertible if and only if $\varphi$ is an automorphism and $\varphi(b)=a$
\end{theorem}

\begin{proof}
If $\varphi$ is an automorphism and $\varphi(b)=a$. Then $\varphi^{-1}(a)=b$. For $f\in H^2_a$, then
$$(C_\varphi f)(b)=f(\varphi(b))=f(a)=0,$$
and for $g\in H^2_b$, then
$$(C_{\varphi^{-1}} g)(a)=g(\varphi^{-1}(a))=g(b)=0.$$
Therefore $C_\varphi: H_a^2\rightarrow H^2_b$ and $C_{\varphi^{-1}}: H_b^2\rightarrow H^2_a$ are well-defined. Notice that for any $z\in\mathbb D$ and $f\in H^2_a$, we obtain that
$$f(z)=f\big(\varphi(\varphi^{-1}(z))\big)=(C_{\varphi^{-1}}C_\varphi f)(z)$$
Then $C_{\varphi^{-1}}C_\varphi=I_{H_a^2}$. In the same way, we get $C_\varphi C_{\varphi^{-1}}=I_{H_b^2}$. So $C_\varphi$ is invertible and the inverse of $C_\varphi$ is $C_{\varphi^{-1}}$.

Conversely, suppose $C_\varphi: H_a^2\rightarrow H^2_b$ is invertible. Let $f(z)=\frac{a-z}{1-\bar{a}z}\in H^2_a$. Then $(C_\varphi f)(z)=f(\varphi(z))\in H_b^2$, that is $f(\varphi(b))=0$. And hence $\varphi(b)=a$. Let $K_w^a$ is the reproducing kernel of $H^2_a$ at the point $w$ and let $K_w^b$ is the reproducing kernel of $H^2_b$ at the point $w$. Let $w_1$ and $w_2$ belong to $\mathbb D$ with $\varphi(w_1)=\varphi(w_2)$. It is not hard to see that
$$K_{\varphi(w_1)}^a=K_{\varphi(w_2)}^a,$$
then
$$
\begin{aligned}
\langle f,C_\varphi^{*}K_{w_1}^a\rangle&=\langle C_\varphi f,K_{w_1}^a\rangle\\
&=f(\varphi(w_1))\\
&=f(\varphi(w_2))\\
&=\langle C_\varphi f,K_{w_2}^a\rangle\\
&=\langle f,C_\varphi^{*}K_{w_2}^a\rangle
\end{aligned}
$$
for $f\in H^2_a$. Since $C_\varphi^{*}$ is invertible, we obtain that $K_{w_1}^a=K_{w_2}^a$, then $w_1=w_2$. So $\varphi$ is injective. Let $\varphi_a(z)=\frac{a-z}{1-\bar{a}z}$ and $\varphi_b(z)=\frac{b-z}{1-\bar{b}z}$. Suppose $\varphi_1=\varphi_a\circ\varphi_b$, then $\varphi_1$ is an automorphism of $\mathbb D$ with $\varphi_1(b)=a$ and $\varphi_1^{-1}=\varphi_b\circ\varphi_a$, which implies that $C^{-1}_{\varphi_1}=C_{\varphi^{-1}_1}$ and $C^{-1}_{\varphi_1}: H_b^2\rightarrow H^2_a$ are well-defined. Let $\psi=\varphi\circ\varphi_1^{-1}$, then
$\psi(a)=\varphi(\varphi_1^{-1}(a))=\varphi(b)=a$
and $C_\psi=C_{\varphi_1^{-1}}C_\varphi: H_a^2\rightarrow H^2_a$ is well-defined. Since $C_\varphi$ and $C_{\varphi_1^{-1}}$ are invertible, we obtain that $C_\psi$ is also invertible. For any $w\in\mathbb D$, notice that the reproducing kernel $K_w^a$ belongs to $H_a^2$. Hence there exist a function $g\in H_a^2$ such that $C_\psi^{*}g=K_w^a$ as $C_\psi^{*}g$ is invertible. Suppose $g(z)=\frac{\bar{\eta}(z-a)}{1-\bar{\eta}(z-a)}\in H_a^2$ for $\eta\in\mathbb D$. Let $f_1$, $f_2$ and $f_1f_2$ belong to $H_a^2$. Then there exist $F_1$, $F_2$ and $F$ belong to $H_a^2$ such that
$$C_\psi F_1=f_1,$$
$$C_\psi F_2=f_2,$$
$$C_\psi F=f_1f_2.$$
For $w\in\mathbb D$, we obtain that
$$
\begin{aligned}
(F_1F_2)(\psi(w))&=F_1(\psi(w))F_2(\psi(w))\\
&=(C_\psi F_1)(w)(C_\psi F_2)(w)\\
&=f_1(w)f_2(w)\\
&=(C_\psi F)(w)\\
&=F(\psi(w)).
\end{aligned}
$$
Thus $F_1F_2=F$. Notice that
$$
\begin{aligned}
\langle f_1f_2,g\rangle&=\langle C_\psi F_1F_2,g\rangle\\
&=\langle F_1F_2,C_\psi^{*}g\rangle\\
&=\langle F_1F_2,K_w^a\rangle\\
&=F_1(w)F_2(w)\\
&=\langle F_1,K_w^a\rangle \langle F_2,K_w^a\rangle\\
&=\langle F_1,C_\psi^{*}g\rangle\langle F_2,C_\psi^{*}g\rangle\\
&=\langle C_\psi F_1,g\rangle \langle C_\psi F_2,g\rangle\\
&=\langle f_1,g\rangle \langle f_2,g\rangle.
\end{aligned}
$$
Hence $f\rightarrow\langle f,g\rangle$ is multiplicative linear functional on $H_a^2$.
Thus we obtain that $g(z)=\frac{\bar{\eta}(z-a)}{1-\bar{\eta}(z-a)}=K_\eta^a(z)$ where $\eta\in\mathbb D$. So $C_\psi^{*}K^a_w=K_w^a$, that is $K^a_{\varphi(\eta)}=K^a_w$. Thus for any $w\in\mathbb D$, there exist a point $\eta\in\mathbb D$ such that $\psi(\eta)=w$. Therefore, $\psi$ is surjective. It is not hard to see that $\psi$ is an automorphism of $\mathbb D$. Notice that $\varphi=\psi\circ\varphi_1$, then $\varphi$ is an automorphism.
\end{proof}

\section{\bf Beurling type invariant subspace of composition operators on Hardy space $H^p$}
Throughout this section, let $\varphi$ be a holomorphic self-map function of $\mathbb D$ and let $\theta\in H^\infty(\mathbb D)$ be an inner function. The symbol $\mathcal Z(f)$ represents the zero set of holomorphic function $f$ and the symbol $mult_f(z)$ represents the multiplicity of $z$ that belongs to set $\mathcal Z(f)$.

Now we are ready to present a important lemma that is basic of this section.

\begin{lemma}\label{lem4}
(1)If $\theta H^p$ is invariant for $C_\varphi$, then $mult_\theta(\alpha)\leq mult_{\theta\circ\varphi}(\alpha)$ for all $\alpha\in \mathcal Z(\theta)$.\\
(2)The quotient $\displaystyle{\frac{\theta\circ\varphi}{\theta}}$ defines a holomorphic function on $\mathbb D$ if and only if $mult_\theta(\alpha)\leq mult_{\theta\circ\varphi}(\alpha)$ for all $\alpha\in \mathcal Z(\theta)$.
\end{lemma}

\begin{proof}
For part 1, if $\theta H^p$ is invariant for $C_\varphi$, that is $C_\varphi(\theta H^p)\subseteq\theta H^p$, there exist a function $f\in H^p$ such that
$$C_\varphi(\theta)=\theta\circ\varphi=\theta f,$$
then $$\mathcal Z(\theta)\subseteq\mathcal Z(\theta\circ\varphi).$$
Hence for any a point $a\in\mathcal Z(\theta)\subseteq\mathcal Z(\theta\circ\varphi)$, then $\theta(\varphi(a))=0$, which is equivalent to $\varphi(a)\in\mathcal Z(\theta)$, Therefore $mult_\theta(\alpha)\leq mult_{\theta\circ\varphi}(\alpha)$.

For part 2, suppose $mult_\theta(\alpha)\leq mult_{\theta\circ\varphi}(\alpha)$. Let $\{a_n\}$ be the zero sequence of $\theta$($\{a_n\}$ may be empty or finite set). Let $B_1$ be the Blaschke product corresponding to the sequence $\{a_n\}$. Thus there exist a holomorphic function $f$ such that
$$\theta=B_1f \text { and } f(z)\neq0 \text{ for all } z\in\mathbb D.$$
If $\alpha$ is a zero of $\theta$, then
$$mult_\theta(\alpha)\leq mult_{\theta\circ\varphi}(\alpha).$$
Thus, there exist a Blaschke product $B_2$ and a holomorphic function $g$ such that
$$\theta\circ\varphi=B_1B_2g \text{ and } g(z)\neq0 \text{ for all } z\in\mathbb D.$$
Then $\displaystyle{\frac{\theta\circ\varphi}{\theta}}=\displaystyle{\frac{B_1B_2g}{B_1f}}=
\displaystyle{\frac{B_2g}{f}}$ is a holomorphic function clearly on $\mathbb D$.
Conversely, if $\displaystyle{\frac{\theta\circ\varphi}{\theta}}$ is a holomorphic function on $\mathbb D$. Suppose to the contrary that $mult_\theta(\alpha)>mult_{\theta\circ\varphi}(\alpha)$ for some $\alpha\in\mathcal Z(\theta)$, then
$$\lim_{z\to \alpha}\displaystyle{\frac{\theta\circ\varphi}{\theta}}=\infty.$$
Thus $\displaystyle{\frac{\theta\circ\varphi}{\theta}}$ is not holomorphic, a contradiction with conditions.
\end{proof}

The inverse of Lemma \ref{lem4}(1) is invalid, that is, $mult_\theta(\alpha)\leq mult_{\theta\circ\varphi}(\alpha)$ can not get $\theta H^p$ is an invariant for $C_\varphi$. And the Lemma \ref{lem4}(1) will be discussed in latter corollary.

The following theorem is the central result of this section.

\begin{theorem}\label{the6}
The following statement are equivalent:\\
(1) $\theta H^p$ is an invariant subspace for $C_\varphi$.\\
(2) $\displaystyle{\frac{\theta\circ\varphi}{\theta}}$ belongs to $\mathcal S(\mathbb D)$.
\end{theorem}

\begin{proof}
$(1)\Rightarrow(2)$: Suppose $\theta H^p$ is invariant under the  composition operator $C_\varphi$. By the Lemma \ref{lem4}, we get that $\displaystyle{\frac{\theta\circ\varphi}{\theta}}$ is a holomorphic function on $\mathbb D$. Since $\theta\circ\varphi\in\theta H^p$, there exist a function $f\in H^p$ such that
$$\theta\circ\varphi=\theta f,$$
then
$$f=\displaystyle{\frac{\theta\circ\varphi}{\theta}}\in H^p.$$
Since $\theta$ is inner function, then $\theta\in H^\infty(\mathbb D)$. By Riesz factorization theorem, there exist a function $g_1\in H^\infty(\mathbb D)$ and a Blaschke product $B_1$ such that
$$g_1(z)\neq0 \text{ for all } z\in\mathbb D \text{ and } \theta=B_1g_1.$$
Again by Riesz factorization theorem, there exist a function $g_2\in H^\infty(\mathbb D)$ and a Blaschke product $B_2$ such that
$$g_2(z)\neq0 \text{ for all } z\in\mathbb D \text{ and } \theta\circ\varphi=B_2g_2.$$
Since $\mathcal Z(\theta)\subseteq\mathcal Z(\theta\circ\varphi)$, there exist a Blaschke product $B_3$ such that
$$B_2=B_1B_3.$$
Then
$$\theta\circ\varphi=B_1B_3g_2.$$
For $g_2\in H^\infty(\mathbb D)$ and $\parallel B_2\parallel_\infty=\parallel B_1B_3\parallel_\infty=1$, as $B_2$ is a inner function, which implies that
$$\parallel\theta\circ\varphi\parallel_\infty=\parallel B_1B_3g_2\parallel_\infty=\parallel g_2\parallel_\infty\leq 1.$$
And $$f=\frac{\theta\circ\varphi}{\theta}=\frac{B_1B_3g_2}{B_1g_1}=\frac{B_3g_2}{g_1}\in H^p.$$
Since $|g_1(e^{it})|=1$ a.e.,taking the radial limit, then
$$|f(e^{it})|=\Big|\frac{B_3(e^{it})g_2(e^{it})}{g_1(e^{it})}\Big|=|g_2(e^{it})| \text{ a.e.,}$$
and hence $f\in H^\infty(\mathbb D)$ and $\parallel f\parallel_\infty=\parallel g_2\parallel_\infty\leq1$. Therefore, $f\in\mathcal S(\mathbb D).$

$(2)\Rightarrow(1)$: Suppose $\displaystyle{\frac{\theta\circ\varphi}{\theta}}\in\mathcal S(\mathbb D)$, there exist a function $f\in\mathcal S(\mathbb D)$ such that
$$\theta\circ\varphi=\theta f.$$
Let $g\in H^p$, then
$$C_\varphi(\theta g)=(\theta\circ\varphi)(g\circ\varphi)=\theta f(g\circ\varphi).$$
And $C_\varphi$ is bounded, then
$$g\circ\varphi\in H^p.$$
Since $f\in H^\infty(\mathbb D)$, we get that $f(g\circ\varphi)\in H^p$. Hence $C_\varphi(\theta g)=\theta f(g\circ\varphi)\in \theta H^p$. This completes the proof of the theorem.
\end{proof}

\begin{corollary}\label{cor1}
Let $B$ be a Blaschke product and let $\varphi$ be a holomorphic self map of $\mathbb D$. Then the following statements are equivalent:\\
(1)$BH^p$ is an invariant subspace under $C_\varphi$.\\
(2)$mult_{B}(w)\leq mult_{B\circ\varphi}(w)$ for all $w$ in $\mathcal Z(B)$.
\end{corollary}

\begin{proof}
$(1)\Rightarrow(2)$: Suppose $BH^p$ is invariant under $C_\varphi$, that is,
$$C_\varphi(BH^p)\subseteq BH^p.$$
By Lemma \ref{lem4}(1), we get easily $mult_B(w)\leq mult_{B\circ\varphi}(w)$ for all $w\in\mathcal Z(B)$.\\
$(2)\Rightarrow(1)$: Suppose $mult_B(w)\leq mult_{B\circ\varphi}(w)$ for all $w$ in $\mathcal Z(B)$. By Riesz factorization theorem, there exist a Blaschke product $B_1$ and a function $g_1\in H^\infty(\mathbb D)$ such that
$$B\circ\varphi=B_1g_1 \text{ and } g_1(z)\neq0 \text{ for all } z\in\mathbb D,$$
since $B\circ\varphi\in H^\infty(\mathbb D)$. According to the condition that $mult_{B_1}(w)=mult_{B\circ\varphi}(w)\geq mult_{B}(w)$, then there exist a Blaschke product $B_2$ such that
$$B_1=BB_2.$$ Thus,
$$B\circ\varphi=BB_2g_1.$$
Since $g_1\in H^\infty(\mathbb D)$ and $\|BB_2\|_{\infty}=1$, as $B_1=BB_2$ is inner function. Then
$$\|g_1\|_\infty=\|BB_2g_1\|_{\infty}=\|B\circ\varphi\|_\infty\leq1.$$
Thus $g_1\in\mathcal S(\mathbb D)$, then
$$\displaystyle{\frac{B\circ\varphi}{B}}=\displaystyle{\frac{BB_2g_1}{B}}=B_2g_1\in\mathcal S(\mathbb D),$$
since $\|B_2g_1\|_\infty\leq\|B_2\|_\infty\|g_1\|_\infty\leq1$. By Theorem \ref{the6}, $BH^p$ is invariant under $C_\varphi$.
\end{proof}

Let $\varphi$ be a holomorphic self-map of $\mathbb D$, we will call $b\in\overline{\mathbb D  }$ a fixed point of $\varphi$ if
$$\lim_{r\rightarrow1^-}\varphi(rb)=b.$$
If $b\in\partial\mathbb D$ is a fixed point, then
$$\lim_{r\rightarrow1^-}\varphi'(rb)=\varphi'(b).$$
Let $\varphi$ be automorphism of $\mathbb D$ and $\varphi$ not be an identity, then $\varphi$ has at least two fixed points in complex plane. $\varphi$ is said to be:\\
(1) elliptic if $\varphi$ has a fixed point in $\mathbb D$.\\
(2) parabolic if $\varphi$ has a unique fixed point in $\partial\mathbb D$.\\
(3) hyperbolic if $\varphi$ has two fixed points in $\partial\mathbb D$.

The following theorem is another form of Theorem \ref{the6}.

\begin{theorem}\label{the7}
Let $\theta$ be an inner function and $\varphi$ be an elliptic automorphism of $\mathbb D$. Then the following statements are equivalent:\\
(1) $\theta H^p$ is an invariant subspace for $C_\varphi$.\\
(2) $\displaystyle{\frac{\theta\circ\varphi}{\theta}}$ is unimodular constant.\\
Moreover, in this case, if $w\in\mathbb D$ is the unique fixed point of $\varphi$, then
$$\displaystyle{\frac{\theta\circ\varphi}{\theta}}\equiv\left\{
\begin{aligned}
\big(&\varphi'(w)\big)^{mult_\theta(w)} &\text{ if } w\in\mathcal Z(\theta)\\
&1  &\text{ otherwise }
\end{aligned}
\right.$$
\end{theorem}

\begin{proof}
$(1)\Rightarrow(2)$: Suppose $\theta H^p$ is an invariant subspace for $C_\varphi$, we get $\displaystyle{\frac{\theta\circ\varphi}{\theta}}\in\mathcal S(\mathbb D)$, there exist a function $f\in\mathcal S(\mathbb D)$ such that $f=\displaystyle{\frac{\theta\circ\varphi}{\theta}}$. Suppose $w\in\mathbb D$ is the unique fixed point of $\varphi$. Let
$$g_w(z)=\frac{w-z}{1-\bar{w}z}$$
for all $z\in\mathbb D$. In the first case, if $w\in\mathcal Z(\theta)$, then there exist an inner function $\theta_1$ such that
$$\theta(z)=\big(g_w(z)\big)^{mult_\theta(w)}\theta_1(z) \text{ and } \theta_1(w)\neq0$$
for $z\in\mathbb D$. Then we get
$$
\begin{aligned}
f&=\displaystyle{\frac{\theta\circ\varphi}{\theta}}\\
&=\displaystyle{\frac{\big(g_w\circ\varphi\big)^{mult_\theta(w)}(\theta_1\circ\varphi)}
{\big(g_w\big)^{mult_\theta(w)}\theta_1}}\\
&=\Big(\displaystyle{\frac{g_w\circ\varphi}{g_w}}\Big)^{mult_\theta(w)}
\big(\frac{\theta_1\circ\varphi}{\theta_1}\big).
\end{aligned}
$$
Then we get
$$
\begin{aligned}
\lim_{z\rightarrow w}\frac{(g_w\circ\varphi)(z)}{g_w(z)}&=\lim_{z\rightarrow w}
\frac{w-\varphi(z)}{1-\bar{w}\varphi(z)}\cdot\frac{1-\bar{w}z}{w-z}\\
&=\lim_{z\rightarrow w}\frac{w-\varphi(z)-|w|^2z+\bar{w}z\varphi(z)}
{w-z-| w|^2\varphi(z)+\bar{w}z\varphi(z)}\\
&=\lim_{z\rightarrow w}\frac{\big(-\varphi'(z)\big)-|w|^2+\bar{w}\varphi(z)+\bar{w}z\varphi'(z)}
{(-1)-|w|^2\varphi'(z)+\bar{w}\varphi(z)+\bar{w}z\varphi'(z)}\\
&=\frac{\varphi'(w)(|w|^2-1)}{|w|^2-1}\\
&=\varphi'(w).
\end{aligned}
$$
Since $w$ is the unique fixed point of $\varphi$, that is $\varphi(w)=w$, which implies that
$$
\begin{aligned}
f(w)&=\Big(\displaystyle{\frac{g_w\circ\varphi(w)}{g_w(w)}}\Big)^{mult_\theta(w)}
\big(\frac{\theta_1\circ\varphi(w)}{\theta_1(w)}\big)\\
&=\big(\varphi'(w)\big)^{mult_\theta(w)}\frac{\theta_1(w)}{\theta_1(w)}\\
&=\varphi'(w)^{mult_\theta(w)}.
\end{aligned}
$$
We get $|\varphi'(w)|=1$ since $\varphi$ is an elliptic automorphism, and hence $|f(w)|=1$. Since $f\in\mathcal S(\mathbb D)$, by the maximum modulus principle, we get $f$ is constant function and $f\equiv f(w)\equiv\varphi'(w)^{mult_\theta(w)}$. The other case, if $w\notin\mathcal Z(\theta)$, that is, $\theta(w)\neq0$, then
$$f(w)=\displaystyle{\frac{\theta\circ\varphi(w)}{\theta(w)}}=\frac{\theta(w)}{\theta(w)}=1.$$
Again by the maximum modulus principle, we get $f$ is constant function and $f\equiv f(w)\equiv1$.

$(2)\Rightarrow(1)$: $\displaystyle{\frac{\theta\circ\varphi}{\theta}}$ is unimodular constant, then $\displaystyle{\frac{\theta\circ\varphi}{\theta}}\in\mathcal S(\mathbb D)$. By Theorem 6, we directly get $\theta H^p$ is an invariant subspace for $C_\varphi$.
\end{proof}

We repeat the proof to get the next two results.

\begin{corollary}\label{cor2}
Let $\varphi$ be a holomorphic self map of $\mathbb D$ and let $w\in\mathbb D$ be the fixed point of $\varphi$. Let $\theta$ be an inner function and suppose that $\theta(w)\neq0$. Then $\theta H^p$ is an invariant subspace for $C_\varphi$ if and only if $\theta\circ\varphi=\theta$.
\end{corollary}

\begin{theorem}\label{the8}
Let $\varphi$ be a parabolic automorphism of $\mathbb D$, every orbit of $\varphi$ is Blaschke summable. Then $B_zH^p\in LatC_\varphi$ for all $z\in\mathbb D$, where $B_z$ is the Blaschke product corresponding to the orbit $\{\varphi_m(z)\}_{m\geq0}$.
\end{theorem}

\begin{proof}
For any $a\in\mathbb D$, $B_z$ is the Blaschke product corresponding to the orbit $\{\varphi_m(a)\}_{m\geq0}$, then $\{\varphi_m(a)\}_{m\geq0}\subseteq\mathcal Z(B_a)$. Since every orbit of $\varphi$ is Blaschke summable, and by Corollary 1, there is clearly that $B_aH^p$ is an invariant subspace for $C_\varphi$, then $B_aH^p\in LatC_\varphi$. Therefore $B_zH^p\in LatC_\varphi$ for all $z\in\mathbb D$.
\end{proof}

\section{\bf Deddens algebra $\mathcal D_{C_\varphi}$ associated to compact composition operators on $H^p$ }
In this section, let $C_\varphi: H^p\rightarrow H^p$ be a bounded linear operator and let $C_\varphi$ be a compact composition operator.

Now we will determine that some types of operators belong to $\mathcal D_{C_\varphi}$ where. In the following two theorems, we will show that composition operator $C_\varphi$, multiplication operator $M_h$ and antiderivative operator $V$ belong to Deddens algebra $\mathcal D_{C_\varphi}$.

\begin{lemma}\label{lem5}
Let $\varphi:\mathbb D\rightarrow\mathbb D$ be an analytic function that satisfies $\varphi(0)=0$. Then the composition operator $C_\varphi$ belongs to the Deddens algebra $\mathcal D_{C_\varphi}$.
\end{lemma}

\begin{proof}
Notice that $C_\varphi$ is a bounded operator, then there exist a constant $M>0$ such that $\parallel C_\varphi\parallel\leq M$. For $f\in H^p$, we obtain that
$$
\begin{aligned}
\parallel C^n_\varphi C_\varphi f\parallel_{H^p}&=\parallel C_\varphi C^n_\varphi f\parallel_{H^p}\\
&\leq \parallel C_\varphi\parallel \parallel C^n_\varphi f\parallel_{H^p}\\
&\leq M\parallel C^n_\varphi f\parallel_{H^p}.
\end{aligned}
$$
Therefore composition operator $C_\varphi$ belongs to $\mathcal D_{C_\varphi}$.
\end{proof}

\begin{theorem}\label{the9}
Let $\varphi:\mathbb D\rightarrow\mathbb D$ be an analytic function that satisfies $\varphi(0)=0$ and let $h\in H^\infty(\mathbb D)$. Then the operators $M_h$ and $V$ belongs to the Deddens algebra $\mathcal D_{C_\varphi}$.
\end{theorem}

\begin{proof}
It is clear that $C_\varphi M_h=M_{h\circ\varphi}C_\varphi$, and then
$$C^2_\varphi M_h=C_\varphi M_{h\circ\varphi}C_\varphi=M_{h\circ\varphi_2}C^2_\varphi,$$
$$C^3_\varphi M_h=C_\varphi M_{h\circ\varphi_2}C^2_\varphi=M_{h\circ\varphi_3}C^3_\varphi.$$
Hence
$$C^n_\varphi M_h=M_{h\circ\varphi_n}C^n_\varphi.$$
Notice that $\parallel M_{h\circ\varphi_n}\parallel=\parallel h\circ\varphi_n\parallel_\infty$ and $\varphi(\mathbb D)\subset\mathbb D$, then we obtain that
$$\parallel h\circ\varphi_n\parallel_\infty\leq\parallel h\parallel_\infty.$$
Let $M_1=\parallel h\parallel_\infty$, it follows that
$$
\begin{aligned}
\parallel C^n_\varphi M_h f\parallel_{H^p}&=\parallel M_{h\circ\varphi_n}C^n_\varphi f\parallel_{H^p}\\
&\leq \parallel M_{h\circ\varphi_n}\parallel \parallel C^n_\varphi f\parallel_{H^p}\\
&\leq \parallel h\parallel_\infty \parallel C^n_\varphi f\parallel_{H^p}\\
&= M_1\parallel C^n_\varphi f\parallel_{H^p}
\end{aligned}
$$
for $f\in H^p$. Therefore, multiplication operator $M_h$ belongs to $\mathcal D_{C_\varphi}$.

Let $f\in H^p$, it is not hard to see that $C_\varphi Vf=VM_{\varphi'}C_\varphi f$, and then
$$
\begin{aligned}
(C_\varphi^2Vf)(z)&=(C_\varphi VM_{\varphi'}C_\varphi f)(z)\\
&=C_\varphi\int_0^z(f\circ\varphi)(w) d(\varphi(w))\\
&=\int_0^z(f\circ\varphi\circ\varphi)(w) d((\varphi\circ\varphi)(w))\\
&=\int_0^z(f\circ\varphi_2)(w)\varphi'_2(w) dw\\
&=(VM_{\varphi'_2}C_\varphi^2f)(z),
\end{aligned}
$$
so that $C_\varphi^nVf=VM_{\varphi'_n}C_\varphi^nf$. By the definition of $\mathcal D_{C_\varphi}$, we get that antiderivative operator $V$ belongs to $\mathcal D_{C_\varphi}$ if and only if there exist a constant $M_2$ such that $\parallel VM_{\varphi'_n}\parallel\leq M_2$. And hence we just have to prove $\parallel VM_{\varphi'_n}\parallel\leq M_2$. It is clear that $\parallel\varphi\parallel_\infty\leq1$ and $\varphi_n$ is a self-map of $\mathbb D$ for $n\in\mathbb N$. For $n\in\mathbb N$, we obtain that $\parallel\varphi_n\parallel_{BMO}\leq1$ where $\parallel\cdot\parallel_{BMO}$ denotes BMO norm \cite{Z}. By a classic result of Pommerenke \cite{AJ,P.D}, we get that
$$\parallel\int_0^z\varphi'_n(w)f(w)dw\parallel_{H^p}\leq M\parallel f\parallel_{H^p}$$
for $M>0$, $f\in H^p$ and $n\in\mathbb N$. Let $f=1\in H^p$, then
$$\parallel VM_{\varphi'_n}1\parallel_{H^p}=\parallel\int_0^z\varphi'_n(w)dw\parallel_{H^p}\leq M.$$
Therefore, antiderivative operator $V$ belongs to $\mathcal D_{C_\varphi}$.
\end{proof}

\begin{lemma}\label{lem6}
If $u$ and $v$ are function that are differentiable $n+1$, then
$$\int u(w)v^{(n)}(w)dw=\sum^n_{j=1}u^{(j-1)}(w)v^{(n-j)}(w)(-1)^{j+1}+(-1)^n\int u^{(n)}(w)v(w)dw.$$
\end{lemma}

The proof process of the next theorem will use Lemma \ref{the6}.

\begin{theorem}\label{the10}
Let $\varphi:\mathbb D\rightarrow\mathbb D$ be an analytic function that satisfies $\varphi(0)=0$. Let $\theta=BS$ be the factorization of the inner function $\theta$ into a Blaschke product $B$ and a singular inner function $S$. If $\alpha\neq0$ is a zero of $B$, then the subspace $\theta H^p$ cannot be invariant for Deddens algebra $\mathcal D_{C_\varphi}$.
\end{theorem}

\begin{proof}
we get from condition that $\varphi(0)=0$ and $\parallel\varphi\parallel_\infty\leq1$. By Schwartz Lemma, we get either
$$|\varphi(z)|<|z| \text{ for all } z\in\mathbb D$$
or
$$\varphi(z)=cz \text{ for some } |c|=1.$$
Suppose $\varphi(z)=cz \text{ for some } |c|=1,$ then $C_\varphi$ is unitary operator. Hence $\mathcal D_{C_\varphi}=\mathcal L(H^p)$ and the theorem holds. Thus, we only consider the first case that $|\varphi(z)|<|z|$ for all $z\in\mathbb D$. Suppose to the contrary that $\theta H^p$ is invariant for $\mathcal D_{C_\varphi}$, that is $\mathcal D_{C_\varphi}(\theta H^p)\subseteq\theta H^p$. And let $\alpha\neq0$ is a zero of $\theta$ with multiplicity $m\geq1$, that is $mult_\theta(\alpha)= m$. Since $\theta H^p$ is invariant for $\mathcal D_{C_\varphi}$ and $C_\varphi\in\mathcal D_{C_\varphi}$, then $\theta H^p$ is invariant for $C_\varphi$. By Lemma \ref{lem4}, we get
$$\theta\circ\varphi(\alpha)=0 \text{ and } mult_{\theta\circ\varphi}(\alpha)\geq mult_\theta(\alpha)=m.$$
Then $\varphi(\alpha)$ is a zero of $\theta$ and $mult_\theta\big(\varphi(\alpha)\big)\geq m$. By parity of reasoning, $\varphi_n(\alpha)$ is a zero of $\theta$ and $mult_\theta\big(\varphi_n(\alpha)\big)\geq m$ for all $n\in\mathbb N$. Since $|\varphi(z)|<|z| \text{ for all } z\in\mathbb D$, the sequence $\{|\varphi_n(\alpha)|\}$ is decreasing. But the sequence cannot have a limit point in $\mathbb D$, otherwise, $\theta$ would be equal to 0 a.e.. Thus, there exist a positive integer $n$ such that $\varphi_n(\alpha)=0$, so $\varphi_n(\alpha)=0$ is a zero of $\theta$ and $mult_\theta(0)\geq m$. Hence there exist a function $g\in H^\infty(\mathbb D)$ such that
$$\theta(z)=z^m(z-\alpha)^mg(z) \text{ and } g(\alpha)\neq0.$$
And the function $\theta(z)z^n\in \theta H^p$ since $z^n\in H^p$ for any positive integer $n$. $\theta H^p$ is invariant for $\mathcal D_{C_\varphi}$ and the antiderivative operator $V\in\mathcal D_{C_\varphi}$, so the subspace $\theta H^p$ is invariant under the operator $V$, that is $V(\theta H^p)\subseteq\theta H^p$. Then
$$V\Big(\theta(z)z^n\Big)=\int^z_0\theta(w)w^ndw$$
belongs to $\theta H^p$ for any $n\in\mathbb N$. Then
$$\int^\alpha_0\theta(w)w^ndw=0 \text{ for } n\in\mathbb N$$
since $\theta(\alpha)=0$. Using the Lemma \ref{lem6}, and the fact that
$$\theta(0)=\theta'(0)=\theta^{(2)}(0)=...=\theta^{(m-1)}(0)=0$$ and
$$\theta(\alpha)=\theta'(\alpha)=\theta^{(2)}(\alpha)=...=\theta^{(m-1)}(\alpha)=0.$$
We get
$$\int^\alpha_0\theta^{(m)}(w)w^{n+m}dw=0 \text{ for } n\in\mathbb N$$
Let $w=s\alpha$, where $s$ is variable, then
$$\int^1_0\theta^{(m)}(s\alpha)(s\alpha)^{n+m}d(s\alpha)=0 \text{ for } n\in\mathbb N,$$
which is equivalent to
$$\int^1_0\theta^{(m)}(s\alpha)s^{n+m}ds=0 \text{ for } n\in\mathbb N.$$
Observing the functions
$$f_1(s)=s^{m}\textmd{Re}(\theta)^{(m)}(s\alpha)$$
and
$$f_2(s)=s^{m}\textmd{Im}(\theta)^{(m)}(s\alpha),$$
it is clearly that $f_1$ and $f_2$ is real-valued and equal to 0 at all moment, then $f_1$ and $f_2$ must be zero functions, so $\theta$ is a polynomial of degree up to $m-1$, contradicting $\theta(z)=z^m(z-\alpha)^mg(z)$. Thus, the assumption is invalid.
\end{proof}

Theorem \ref{the10} states that if $\theta H^p$ is invariant under the $\mathcal D_{C_\varphi}$, then $\theta(z)=z^kS(z)$ , where $k$ belongs to $\mathbb N$ and some singular inner function $S$. Then, we will prove that the $S$ must be a constant.

\begin{theorem}\label{the11}
Let $\varphi:\mathbb D\rightarrow\mathbb D$ be a analytic function that satisfies $\varphi(0)=0$. Let $S$ be a singular inner function and $n$ be a non-negative integer. If the subspace $Sz^nH^p$ is invariant for Deddens algebra $\mathcal D_{C_\varphi}$, then $S$ must be a constant.
\end{theorem}

\begin{proof}
Since $S$ is a singular inner function, there exist a finite positive measure $\mu$ such that $\mu$ is singular corresponding to Lebesgue measure and
$$S(z)=\exp\Big\{-\int^{2\pi}_0\frac{e^{it}+z}{e^{it}-z}d\mu(t)\Big\}.$$
Let $t\in[0,2\pi]$ and $e^{it}\in\mathbb T$, viewing $\mu$ as a measure on $\mathbb T$, then there exist a set $E$ of $\mu$ with $\mu(E)=0$ such that
$$\lim_{r\uparrow1}S(r\zeta)=0$$
for all $\zeta\in\mathbb T\setminus E$. Notice for any $f\in Sz^n H^p$ that $f$ extends to all inner functions which is part of $S$, and let $I_f$ be the inner factor of $f$. Since $Sz^n H^p$ is invariant for $\mathcal D_{C_\varphi}$ and the antiderivative operator $V\in\mathcal D_{C_\varphi}$, we get
$$V(Sz^n H^p)\subseteq Sz^n H^p.$$
And the function
$$f_m(z)=S(z)z^nz^m=S(z)z^{n+m}$$
belongs to $Sz^n H^p$ for $m\in\mathbb N_0$. Then
$$Vf_m\in Sz^n H^p.$$
Let $F_m=Vf_m$, then
$$\lim_{r\uparrow1}I_{F_m}=0 \text{ for all } \zeta\in\mathbb T\setminus E
\text{ and all } m\in\mathbb N_0.$$
Let $O_{F_m}$ be the outer factor of $F_m$, then
$$O_{F_m}(r\zeta)=\exp\Big\{\frac{1}{2\pi}\int^{2\pi}_0\frac{e^{it}+r\zeta}{e^{it}-r\zeta}
\log|F_m(e^{it})|d\mu(t)\Big\}.$$
Thus,
$$|O_{F_m}(r\zeta)|=\exp\Big\{\frac{1}{2\pi}\int^{2\pi}_0P_r(t-t_0)
\log|F_m(e^{it})|d\mu(t)\Big\},$$
where $\zeta=e^{it_0}$ and $P_r=\displaystyle{\frac{1-r^2}{|1-r^2e^{i(t-t_0)}|^2}}$ is Poisson kernel that is given by transformation of $\displaystyle{\frac{e^{it}+r\zeta}{e^{it}-r\zeta}}$. It is clearly that $\parallel F_m\parallel_\infty\leq1$. Then $\log|F_m(e^{it})|\leq0$ for $t\in[0,2\pi]$ a.e.. And Hence
$$|O_{F_m}(r\zeta)|\leq1 \text{ for all } \zeta\in\mathbb T\setminus E
\text{ and all } m\in\mathbb N_0,$$
as Poisson kernel $P_r>0$. Since $F_m=I_{F_m}O_{F_m}$, we obtain that
\begin{equation}\label{eq4.1}
\lim_{r\uparrow1}F_m(r\zeta)=0 \text{ for all } \zeta\in\mathbb T\setminus E
\text{ and all } m\in\mathbb N_0.
\end{equation}
Let $w=s\zeta$, where $s$ is variable, then
$$
\begin{aligned}
F_m(r\zeta)&=Vf_m(r\zeta)\\
&=V\big(S(r\zeta)(r\zeta)^{n+m}\big)\\
&=\int_0^{r\zeta}S(w)w^{n+m}dw\\
&=\int_0^{r}S(s\zeta)s^{n+m}\zeta^{n+m+1}ds\\
&=\zeta^{n+m+1}\int_0^{r}S(s\zeta)s^{n+m}ds.
\end{aligned}
$$
Now (\ref{eq4.1}) implies that the functions
$$F_1(s)=\textmd{Re}S(s\zeta)s^n$$
and
$$F_2(s)=\textmd{Im}S(s\zeta)s^n$$
are continuous real-valued and equal to 0 at all moment. So $F_1$ and $F_2$ must be zero function.
Hence $S(s\zeta)\equiv0$ for $\zeta\in\mathbb T\setminus E$ and all $s\in(0,1)$. Then $E=\mathbb T$. It is clearly that $\mu(\mathbb T)=0$. Thus, $S$ must be a constant.
\end{proof}

Our final step is to prove that $z^nH^p$ is indeed invariant for $\mathcal D_{C_\varphi}$. Previously, we need to prove the following lemma.

\begin{lemma}\label{lem7}
Let $\varphi:\mathbb D\rightarrow\mathbb D$ be a nonzero analytic function such that $C_\varphi$ is a compact composition operator. For any $k\in\mathbb N_0$, let $J_k$ be the set of positive integers $m$ such that $T(z^m)=z^kf(z)\in z^kH^p$ and $f(0)\neq0$ for some $f\in H^p$ and some $T\in\mathcal D_{C_\varphi}$.Then, $J_k$ must be a finite set for every $k\in\mathbb N_0$.
\end{lemma}

\begin{proof}
Suppose to contrary that there exist a integer $k\in\mathbb N_0$ such that $J_k$ is an infinite set. Let $m\in J_k$ and $T\in\mathcal D_{C_\varphi}$ such that
$$T(z^m)=z^kf(z)$$
for some $f\in H^p$. Since antiderivative operator $V\in\mathcal D_{C_\varphi}$, by the definition of Deddens algebra $\mathcal D_{C_\varphi}$, we get
$$TV\in\mathcal D_{C_\varphi}.$$
It is clearly that $TVz^{m-1}\in z^kH^p$, which implies that if $m\in J_k$, then $m-1\in J_k$. Since $J_k$ is an infinite set, we obtain that
$$J_k=\mathbb N_0.$$
Thus, subspace $z^mH^p$ is not invariant for Deddens algebra $\mathcal D_{C_\varphi}$ for $m>k$. In other words, there exist a integer $N\in\mathbb N$ such that $z^NH^p$ is invariant for $\mathcal D_{C_\varphi}$, but $z^nH^p$ is not invariant for $\mathcal D_{C_\varphi}$ for $n>N$. Then $z^NH^p$ is invariant under $C_\varphi$ since $C_\varphi\in\mathcal D_{C_\varphi}$. Let operator $A$ be the restriction of $C_\varphi$ to $z^NH^p$, that is,
$$A=C_\varphi|z^NH^p,$$
then $A$ is a compact operator. Let the nontrivial invariant subspace of $\mathcal D_A$ be $\mathcal M$, then
$$\mathcal M\subset z^NH^p.$$
Let $T\in\mathcal D_{C_\varphi}$ and let operator $B$ be the restriction of $T$ to $z^NH^p$, that is,
$$B=T|z^NH^p,$$
then $B\in\mathcal D_A$ since $z^NH^p$ is invariant for $\mathcal D_{C_\varphi}$. Thus, the subspace $\mathcal M$ is invariant for $B$. And the invariant subspace of operator $T$ must be the form $z^nH^p$ for some $n>N$. The contradiction shows that $J_k$ is a finite set.
\end{proof}

The next theorem is the most important theorem in this section.

\begin{theorem}\label{the12}
Let $\varphi:\mathbb D\rightarrow\mathbb D$ be a nonzero analytic function that satisfied $\varphi(0)=0$ and $\parallel\varphi\parallel_\infty<1$. Then $z^nH^p$ is nontrivial invariant subspace for $\mathcal D_{C_\varphi}$.
\end{theorem}

\begin{proof}
Notice that $\parallel\varphi\parallel_\infty<1$. Let $\delta=\parallel\varphi\parallel_\infty<1$. Taking the Schwartz Lemma applied to function $\varphi/\delta$, then
$$\varphi(z)/\delta<|z| \text{ for all } z\in\mathbb D,$$
that is,
$$\varphi(z)<\delta|z| \text{ for all } z\in\mathbb D.$$
By parity of reasoning, we get
$$|\varphi_n(z)|<\delta^n|z| \text{ for all } z\in\mathbb D \text{ and } n\in\mathbb N.$$
Taking the radial limits, then
$$|\varphi_n(e^{i\theta})|<\delta^n \text{ for a.e. } \theta\in[0,2\pi] \text{ and } n\in\mathbb N.$$
Let $E$ be the subset of $\mathbb T$ such that there exist at least one function $\varphi_n$ has no radial limit, then the Lebesgue measure of subset $E$ is equal to 0 and the sequence $\{\varphi_n\}$ is convergent on $\overline{\mathbb D}\backslash E$ as well as its limit is 0. The next is to prove the theorem. Suppose to contrary that $z^nH^p$ is not nontrivial invariant subspace for $\mathcal D_{C_\varphi}$, which is equivalent to there exist non-negative integer $k,m$ and function $f\in H^p$ where $m>k$ such that
$$T(z^m)=z^kf(z) \text{ and } f(0)\neq0.$$
By the definition of $\mathcal D_{C_\varphi}$, there exist $M>0$ such that
$$\parallel C_\varphi^nT(z^m)\parallel^p_{H^p}\leq M\parallel C_\varphi^n(z^m)\parallel^p_{H^p}$$
which is equivalent to
$$\int_0^{2\pi}|\varphi_n(e^{i\theta})|^{kp}|f\circ\varphi_n(e^{i\theta})|^pd\theta\leq M
\int_0^{2\pi}|\varphi_n(e^{i\theta})|^{mp}d\theta$$
for all $n\in\mathbb N$. Since $\{\varphi_n\}$ converges uniformly to 0 on $\overline{\mathbb D}\backslash E$, there exist $N\in\mathbb N$ such that
$$|f\circ\varphi_n(e^{i\theta})|^p\geq\frac{|f(0)|^p}{2}$$
for all $n>\mathbb N$ and all $e^{i\theta}\in\overline{\mathbb D}\backslash E$. Then
\begin{equation}\label{eq4.2}
\frac{|f(0)|^p}{2}\int_0^{2\pi}|\varphi_n(e^{i\theta})|^{kp}d\theta\leq M
\int_0^{2\pi}|\varphi_n(e^{i\theta})|^{mp}d\theta
\end{equation}
for all $n\in\mathbb N$. And
\begin{equation}\label{eq4.3}
\begin{aligned}
\int_0^{2\pi}|\varphi_n(e^{i\theta})|^{mp}d\theta&=
\int_0^{2\pi}|\varphi_n(e^{i\theta})|^{kp}|\varphi_n(e^{i\theta})|^{(m-k)p}d\theta\\
&\leq\int_0^{2\pi}|\varphi_n(e^{i\theta})|^{kp}|\delta^n|^{(m-k)p}d\theta
\end{aligned}
\end{equation}
Taking (\ref{eq4.3}) in (\ref{eq4.2}), we get
$$\frac{|f(0)|^p}{2}\int_0^{2\pi}|\varphi_n(e^{i\theta})|^{kp}d\theta\leq M
|\delta|^{(m-k)np}\int_0^{2\pi}|\varphi_n(e^{i\theta})|^{kp}d\theta$$
for all $n\in\mathbb N$. Thus,
$$\frac{|f(0)|^p}{2}\leq M|\delta|^{(m-k)np} \text{ for all } n\in\mathbb N.$$
And hence $\displaystyle{\frac{|f(0)|^p}{2}}\leq M|\delta|^{(m-k)np}\rightarrow0$ when $n\rightarrow\infty$. The contradiction show that $T(z^m)\in z^mH^p$ for all $m\in\mathbb N$. Therefore $z^nH^p$ is a nontrivial invariant subspace for $\mathcal D_{C_\varphi}$.
\end{proof}

\vspace{0.3cm}
{\bf Acknowledgments: }
 The authors are thankful to MuthuKumar for his valuable suggestions.

\end{document}